\documentclass[12pt]{amsart}  
\usepackage{amssymb}
\usepackage{subcaption}
\usepackage{amsthm}
\usepackage{amsmath}
\usepackage{graphicx}
\usepackage{mathabx}
\usepackage{hyperref}
\usepackage{relsize}
\usepackage{pinlabel}
\usepackage[margin = 1.3in]{geometry}

\newtheorem{Thm}{Theorem}
\newtheorem{Prop}[Thm]{Proposition}

\newtheorem{Conj}[Thm]{Conjecture}

\theoremstyle{definition}

\theoremstyle{remark}

\theoremstyle{definition}

\title{Doubly Slice Odd Pretzel Knots}
\author{Clayton McDonald}
\address{Department of Mathematics, Boston College, 140 Commonwealth Avenue, Chestnut Hill, MA 02467
}
\email[Clayton McDonald]{mcdonafi@bc.edu}

\begin{document}

\maketitle

\begin{abstract}
We prove that an odd pretzel knot is doubly slice if it has $2n+1$ twist parameters consisting of $n+1$ copies of $a$ and $n$ copies of $-a$ for some odd integer $a$. Combined with the work of Issa and McCoy, it follows that these are the only doubly slice odd pretzel knots.
\end{abstract}
\vspace{0.3 in}
\section{Introduction}
Issa and McCoy \cite{issamccoy} give an almost complete classification of which odd pretzel knots $K$ are doubly slice by obstructing embeddings of their branched double covers $\Sigma_2(S^3,K)$ into $S^4$, generalizing work of Donald \cite{donald}. We denote a pretzel knot with $a_i$ twists in its $i$th twist box as $P(a_1,a_2,\dots,a_{n})$, with an odd pretzel knot being such a knot with only odd numbers of twists in each twist box and an odd number of parameters. Recall that a knot $K$ is (smoothly) \textbf{doubly slice} if it is the intersection of a (smoothly) unknotted 2-sphere in $S^4$ with a meridional $S^3$ (i.e.~an $S^3$ that bounds a $B^4$ on both sides). Alternatively, $K$ is doubly slice if there are two slice discs for $K$ whose concatenation is an unknotted $S^2 \subset S^4$.

\begin{Thm} \cite[Theorem 1.11]{issamccoy}
\label{imc}
If $K$ is an odd pretzel knot, then the following are equivalent:

\begin{enumerate}
\item $\Sigma_2(S^3,K)$ embeds smoothly in $S^4$,

\item $K$ is a mutant of a smoothly doubly slice pretzel knot, and
\item $K$ is a mutant of $P(a,-a,a,-a,\dots,a)$ for some odd $a$.
\end{enumerate}
\end{Thm}

More specifically, Issa and McCoy establish in their proof of Theorem \ref{imc} that all mutants of $P(a,-a,a,-a,\dots,a)$ that are odd pretzel knots are odd pretzels whose parameters can be notated as a permutation of those of $P(a,-a,a,-a,\dots,a)$ (i.e.~a pretzel with exactly $n+1$ twist parameters equal to $a$ and $n$ twist parameters equal to $-a$).
They observe in their proof of Theorem \ref{imc} that the $P(a,-a,a,-a,\dots,a)$ are doubly slice, so what remains for a full classification is to check their mutants. We strengthen the conclusion of Theorem \ref{imc} by showing that all of these odd pretzel mutants of $P(a,-a,a,-a,\dots,a)$ are also doubly slice.
\begin{Thm}
\label{mainthm}
Any odd pretzel mutant of the odd pretzel knot $P(a,-a,a,-a,\dots,a)$ is doubly slice.
\end{Thm}
We exhibit the double slicings of these mutants via particular sets of band attachments on these pretzel knots.

Combining Theorems \ref{imc} and \ref{mainthm}, we immediately obtain:

\begin{Thm}
For $K$ an odd pretzel knot, the following are equivalent:
\begin{enumerate}
\item $\Sigma_2(S^3,K)$ embeds in $S^4$,
\item $K$ is a doubly slice pretzel knot, and
\item $K$ is a mutant of $P(a,-a,a,-a,\dots,a)$ for some odd $a$. \qed
\end{enumerate}
\end{Thm}


\section{The Main Argument}

For a set of pairwise disjoint bands $S$ attached to a knot $K$, we denote the knot/link resulting from applying those band moves as $K*S$. 
Let $U_{n}$ denote the $n$ component unlink. For our construction of the double slicing sphere, we use the following criterion of Donald. This criterion is derived by repeated application of a theorem of Scharlemann \cite{sch} stating that the only band move up to isotopy on a split link that results in the unknot is the trivial band move on $U_2$, i.e.~a planar band move on a planar two component unlink. 
\begin{Thm}\cite[Corollary 2.5]{donald}
\label{crit}
Let $\mathcal{A} = \{A_1,\dots,A_n\}$ and $\mathcal{B} = \{B_1,\dots,B_n\}$ be two sets of $n$ bands for $K$ such that:
\begin{enumerate}
\item $K*\mathcal{A} = K*\mathcal{B} = U_{n+1}$,

\item and $K*\mathcal{A}*B_1*\dots*B_k = K*A_1*\dots*A_k*\mathcal{B} = U_{n+1-k}$, $\forall k$.

\end{enumerate}
Then $K$ is doubly slice. 
\end{Thm} 
\begin{proof}[Sketch of Proof]

By imagining the $\mathcal{A}$ bands in the past and the $\mathcal{B}$ bands in the future, we can construct a cobordism in $S^3 \times [-1,1]$ from $K*\mathcal{A} = U_{n+1} \subset S^3 \times \{-1\}$ to $K*
\mathcal{B} = U_{n+1} \subset S^3 \times \{1\}$ with $K\subset S^3 \times \{0\}$ as a cross section.
By capping off this cobordism with unknotted discs in $B^4$, we obtain a Morse function on the associated sphere in $S^4$, with index 0 critical points corresponding to each unlinked component of $K*\mathcal{A}$, index 1 critical points to each band, and index 2 critical points to each unlink component of $K*\mathcal{B}$. For all $k$, $K*\mathcal{A}*B_1*\dots*B_k$ is an unlink, so by Scharlemann's theorem, each successive band attachment is a trivial attachment between two of the unlink components. We can cancel these corresponding 0-1 critical point pairs in the Morse function via isotopy of the 2-sphere. Similarly, if we turn the Morse function upside down and attach $A_1,\dots,A_k$ to $K*\mathcal{B}$, we get an unlink, so at each step the band attachments are also trivial by Scharlemann's theorem. Therefore, we can remove all of the index 1 critical points via isotopy, so our sphere must be unknotted. This sphere had $K$ as a cross section, so $K$ is doubly slice.
\end{proof}
To prove Theorem \ref{mainthm}, we examine the double slicing bands of $P(a,-a,a,-a,\dots,a)$ as in Theorem \ref{crit} and see how we can naturally generalize to the mutants corresponding to various permutations of the parameters. For the standard double slicing of $P(a,-a,a,-a,a)$, the bands are as in Figure \ref{fig:std}. Each of these bands effectively cancels two adjacent twist boxes of the pretzel that have inverse numbers of twists, at the cost of adding an extra unknotted, unlinked component (see Figure \ref{localmove}).

The $\mathcal{A}$ bands (on the outside) cancel the $-a$'s with the $a$'s counterclockwise adjacent to them, whereas the $\mathcal{B}$ bands (on the inside) cancel $-a$'s with the $a$'s clockwise adjacent to them.
To achieve such cancellations, we attach flat bands with feet directly outside the pair of twist boxes we are cancelling, like the bottom band in Figure \ref{localmove}.
The $\mathcal{A}$ bands are attached flatly in the unbounded region so that they do not cross, and the $\mathcal{B}$ bands are attached flatly in the central region so that they do not cross. In the case of a larger number of twist boxes, we simply extend this pattern to create a set of $\mathcal{A}$ and $\mathcal{B}$ bands for $P(a,-a,a,-a\dots,a)$.
As seen in Figure \ref{localmove}, the cancellations from one color of band can be done without moving the other bands. After these cancellations, the diagram of $K*\mathcal{A}$ is the disjoint union of a planar $U_n$ and the standard diagram of $P(a)$. Furthermore, the $\mathcal{B}$ bands are planar. 
Consequently, $K*\mathcal{A}*B_1*\dots*B_k$ and
$K*\mathcal{B}*A_1*\dots*A_k$ are unlinks for all $k$. Moreover, Scharlemann's theorem is unnecessary in this case to certify the conclusion of Theorem \ref{crit}, as every $\mathcal{B}$ band attachment to $K*\mathcal{A}$ is a trivial band attachment, as well as every $\mathcal{A}$ band attachment to $K*\mathcal{B}$. All that remains is to make sure each stage has the correct number of components. This is equivalent to verifying that every $\mathcal{A}$ band attachment to $K*\mathcal{B}$ is a fusion band, i.e. one that joins two distinct components, and likewise for every $\mathcal{B}$ band attachment to $K*\mathcal{A}$. It suffices to show that $K*\mathcal{A}*\mathcal{B}$ is the unknot, as we know $K*\mathcal{A} = K*\mathcal{B} = U_{n+1}$. The rest would follow because an oriented band move must change the number of components by exactly one.

\begin{figure}[htbp]
\begin{subfigure}{0.4\textwidth}
    \centering
    \includegraphics[width=0.8\textwidth]{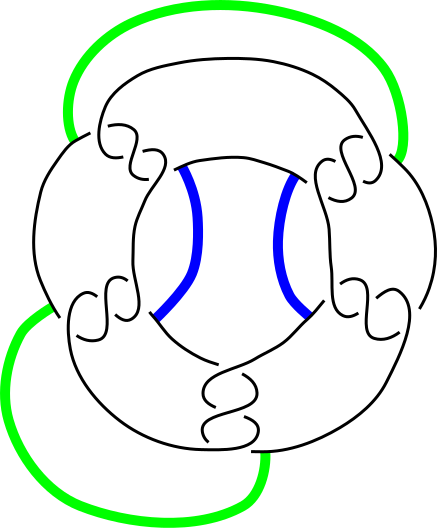}
\end{subfigure}
\hspace{1cm}
\begin{subfigure}{0.4\textwidth}
    \centering
    \labellist
    \small\hair 2pt
    \pinlabel $-$ at 35 100
    \pinlabel $+$ at 155 10
    \pinlabel $+$ at 280 100
    \pinlabel $+$ at 85 245
    \pinlabel $-$ at 240 245
    \endlabellist
    \includegraphics[width=0.8\textwidth]{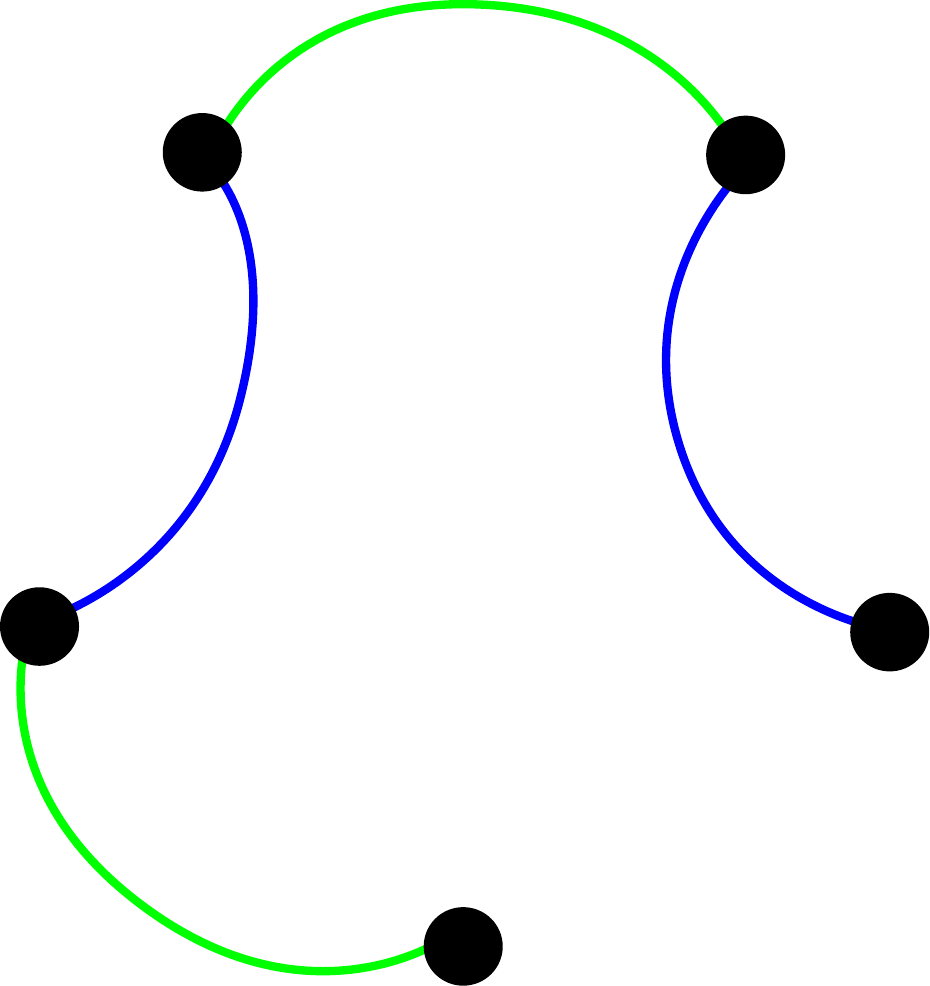}
\end{subfigure}
\caption{The double slicing bands of $P(3,-3,3,-3,3)$ and the auxiliary graph $G(+,-,+,-,+)$.}
\label{fig:std}
\end{figure}
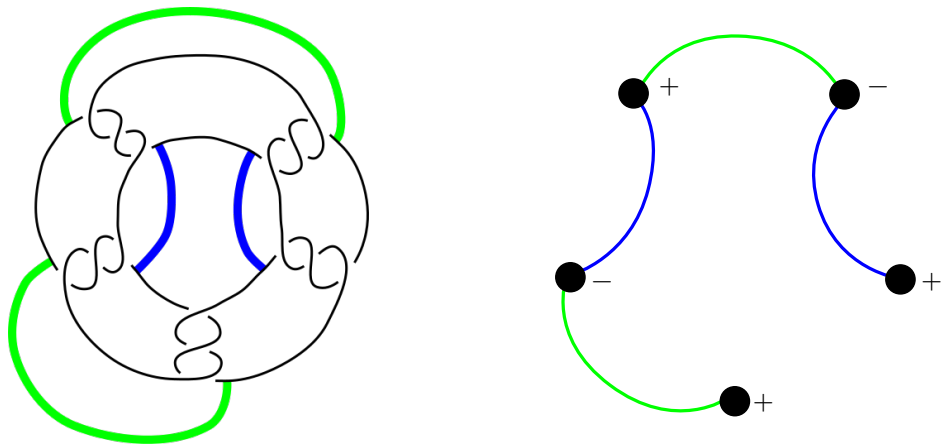
We will use a similar set of bands and do a similar set of schematic cancellations of twist boxes for the more general permuted $a$,$-a$'s.
For the general case, our set of $\mathcal{A}$ bands is obtained by the following iterative procedure. As in the model example, attach the $\mathcal{A}$ bands as flat bands contained in the unbounded region of the standard planar diagram of the pretzel. First, add a band cancelling any $(-a)$-twist box with the $a$ counterclockwise adjacent to it, if there is one. Then, add bands that would cancel $(-a)$-twist boxes that have $a$-twist boxes counterclockwise adjacent to them after doing the cancellations from the previous step. Iterate this process until all of the $(-a)$-twist boxes are cancelled, leaving the one stranded pretzel knot $P(a)$, which is unknotted, along with one planar unknotted component for each band. The set of all $n$ such bands used is $\mathcal{A}$.

\begin{figure}[htbp]
    \centering
    \includegraphics[width = 0.8\textwidth]{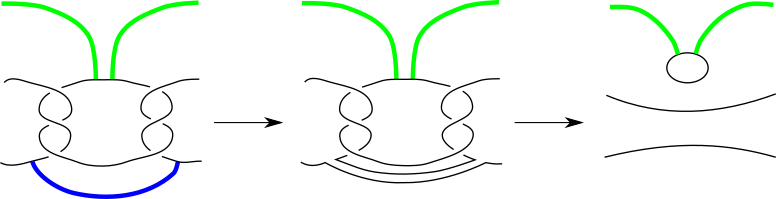}
    \caption{The bottom band move and a local isotopy of the resulting link while fixing the top bands. Note that this works for any odd number of twists.}
    \label{localmove}
\end{figure}

Our set of $\mathcal{B}$ bands is obtained by a similar process. As before we attach the $\mathcal{B}$ bands as flat bands in the central region of the planar diagram (see Figure \ref{fig:modelex}).
In this case, we recursively cancel our $(-a)$-twist boxes with the clockwise adjacent $a$-twist boxes. By the same reasoning as for $\mathcal{A}$, this set of band attachments yields an unlink. 
Note that in the case of $K = P(a,-a,a,-a,\dots,a)$, this procedure outputs the same sets of bands as in the model example.
To prove these band sets give us double slicings, we first ensure that $K*\mathcal{A}*B_1*\dots*B_k$ and $K*A_1*\dots*A_k*\mathcal{B}$ are unlinks for all $k$. This follows by the same logic as in the model example, with the cancellations being done in stages corresponding to the iterations in the definition of the bands. As in the model case, the cancellation of the handles can be seen directly from the diagram without an appeal to Scharlemann's theorem.

Thus we have verified condition $(1)$ and part of condition $(2)$ of Theorem \ref{crit}. In order to apply Theorem \ref{crit}, it remains to check that at each stage we not only have an unlink, but one with the correct number of components.
\begin{figure}[htbp]
\begin{subfigure}{0.4\textwidth}
    \centering
    \includegraphics[width=0.8\textwidth]{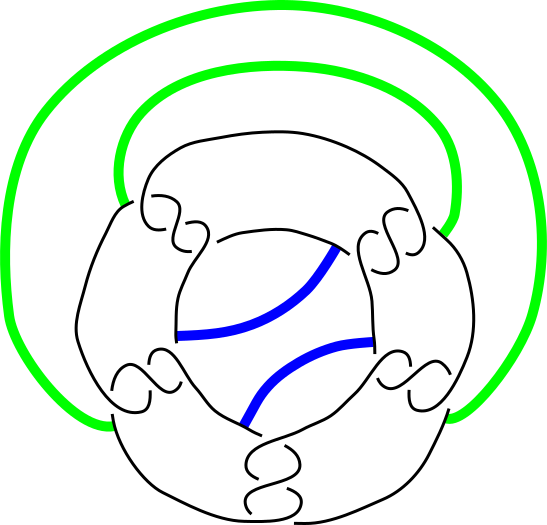}
    
\end{subfigure}
\hspace{1cm}
\begin{subfigure}{0.4\textwidth}
    \centering
    \labellist
    \small\hair 2pt
    \pinlabel $+$ at 35 100
    \pinlabel $+$ at 155 10
    \pinlabel $-$ at 280 100
    \pinlabel $+$ at 85 245
    \pinlabel $-$ at 240 245
    \endlabellist
    \includegraphics[width=0.8\textwidth]{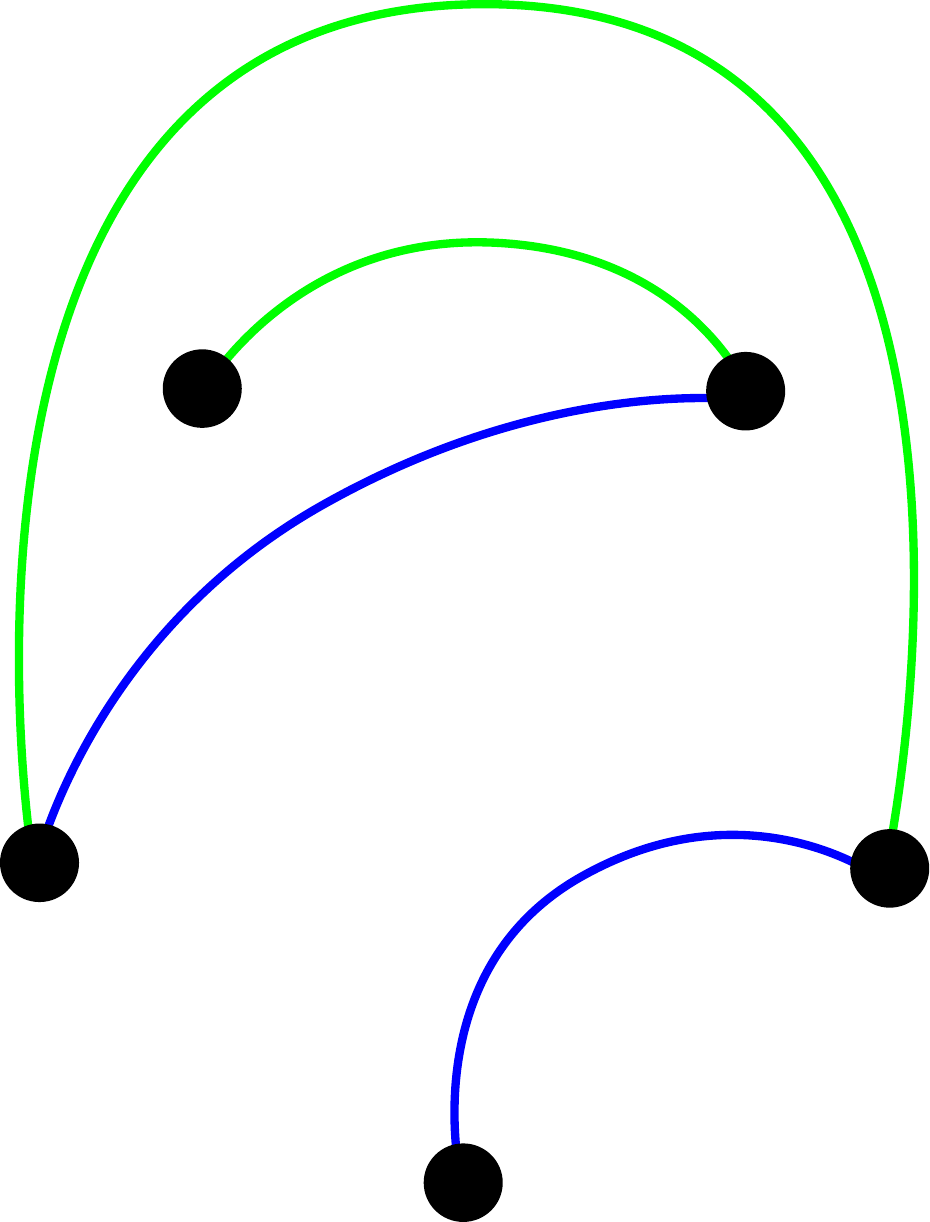}
\end{subfigure}
\caption{The double slicing bands of $P(3,3,3,-3,-3)$ and the auxiliary graph $G(+,+,+,-,-)$.}
\label{fig:modelex}
\end{figure}
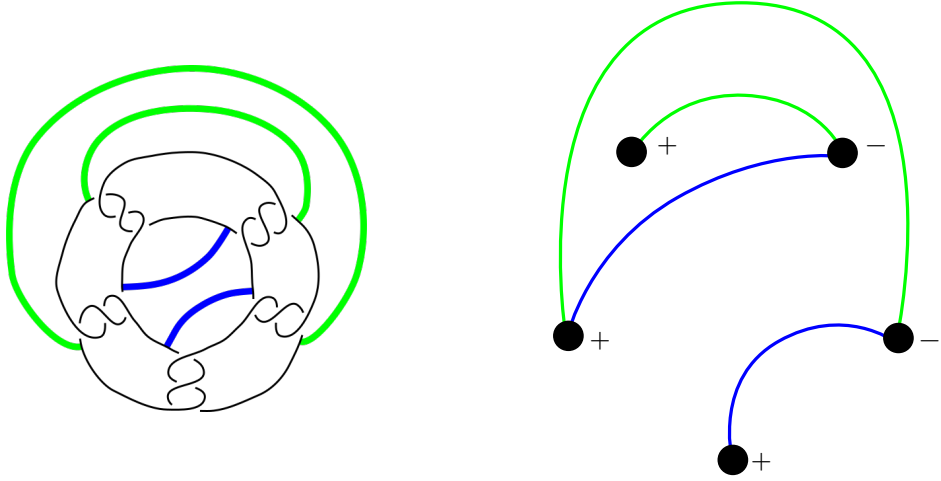

For $n \in \mathbb{N}$ let $\underline{\mu}$ denote a $(2n+1)$-tuple $(\mu_1,\dots,\mu_{2n+1})\in \{\pm\}^{2n+1}$ with $(n+1)$ copies of $+$ and $n$ copies of $-$ among our $\mu_i$'s.
To encode the combinatorics of the band attachments corresponding to $K = P(\mu_1a,\dots,\mu_{2n+1}a)$, we construct a graph
$G(\underline{\mu})$.   The vertex set $V_K$ consists of $2n+1$ cyclically ordered vertices, one for each twist box, and has a natural partition into $V_{+}$ and $V_{-}$, corresponding to the $a$ and $(-a)$-twist boxes respectively. The edge set is partitioned into two sets of $n$ edges, one set $E_A$
with adjacency according to the pairings of the twist boxes via the $\mathcal{A}$ bands, and one set $E_B$
with adjacency according to the $\mathcal{B}$ bands. Note that the connected components of $K*\mathcal{A}$ correspond exactly to the connected components of $(V_K,E_A)$. 
Each pair of vertices in $(V_K,E_A)$ joined by an edge corresponds to the unknotted component formed by the cancellation of the corresponding twist boxes, and the one unpaired vertex corresponds to the remaining $P(a)$ component.

The foot of a $\mathcal{B}$ band attaches clockwise adjacent to the $a$-twist box it is cancelling, and counterclockwise adjacent to the $(-a)$-twist box it is cancelling. The foot of an $\mathcal{A}$ band attaches clockwise adjacent to the $(-a)$-twist box it is cancelling, and counterclockwise adjacent to the $a$-twist box it is cancelling.
Therefore, a $\mathcal{B}$ band attached directly adjacent to a given twist box attaches to the unknotted component formed by the $\mathcal{A}$ band cancelling that twist box. On the other hand, the $\mathcal{B}$ band will attach to the remaining $P(a)$ if its edge in $E_B$ attaches to the unique unpaired vertex in $(V_K,E_A)$. Therefore, if there are no cycles in $G(\underline\mu)$, then every $\mathcal{B}$ band is attaching two separate components of $K *  \mathcal{A}$, so every band is a fusion band. 
Because $G(\underline\mu)$ has $2n+1$ vertices and $2n$ edges, it has no cycles if and only if it is connected, which would require it to be a tree.
This would mean that condition $(2)$ of Theorem \ref{crit} is satisfied for $K$, $\mathcal{A}$, and  $\mathcal{B}$, so $K$ would be doubly slice.
In the case of $K = P(a,-a,a,-a,\dots,a)$, elements of $V_{-}$ are connected to cyclically adjacent elements of $V_{+}$, so connectedness of $G(\underline\mu)$ is clear (compare the two sides of Figure~\ref{fig:std}). It then suffices for the proof of Theorem \ref{mainthm} in the general case to prove the following:


\begin{Prop}
\label{graphprop}
For any $\underline{\mu} = (\mu_1,\dots,\mu_{2n+1})$, $\mu_i \in \{\pm\}$, $n+1$ of which are $+$, $G(\underline\mu)$ is a path. 
\end{Prop}
\begin{proof}[Proof of Proposition]
Because every vertex has degree at most two and there are $2n$ edges for $2n+1$ vertices, it is equivalent to show that there are no cycles in $G(\underline\mu)$. 

Orient the edges of the graph such that $E_A$ edges travel from $V_-$ to $V_+$ and $E_B$ edges travel from $V_+$ to $V_-$. Next, define $V(e)$ for an edge $e$ as the subset of $V$ containing every vertex strictly between $e$'s endpoints starting at $e$'s initial endpoint and moving counterclockwise through the cyclic order.

For $v_- \in V_-$ and $e \in E_A$ attached to $v_-$, we claim that the other endpoint of $e$ is the first $v_+ \in V_+$ moving counterclockwise from $v_-$ such that $V(e)$ would contain the same number of $V_+$ and $V_-$ vertices.
We prove this claim by induction on the counterclockwise distance from a given vertex $v_- \in V_-$ to $v_+$, the $V_+$ vertex it is paired to via $E_A$. The base case of a counterclockwise adjacent $v_+$ is immediate from the recursive definition. Now suppose that $v_-$ is not cyclically adjacent to $v_+$, and that the claim holds for any such graph and pair $(v_-,v_+)$ therein whose cyclic distance is smaller.
The closest vertex to $v_-$ in this gap must also be in $V_-$, otherwise this would be the vertex $v_-$ pairs to.
Similarly, the closest vertex to $v_+$ in this gap must be in $V_+$. Therefore, there must be a point in this gap where there is a $V_-$ vertex with a $V_+$ vertex counterclockwise adjacent to it. The two corresponding twist boxes will cancel in our recursive process, leaving a graph with smaller gap between the two marked vertices, so by the inductive step we know that the recursive adjacency definition would pair them together. 
Similarly, for $v_+ \in V_+$ and $e \in E_B$ attached to $v_+$ (if there is such an edge)  the other endpoint of $e$ is the first $v_- \in V_-$ moving counterclockwise from $v_+$ such that $V(e)$ would contain an equal number of $V_+$ and $V_-$ vertices.

Now assume by way of contradiction that there exists a cycle in  $G(\underline\mu)$. 
Because any vertex can have at most one $E_A$ edge and one $E_B$ edge attached to it, any cycle must alternate $E_A$ and $E_B$ edges, meaning that any cycle must alternate between $V_+$ and $V_-$ vertices.
For every edge $e$, $V(e)$ contains an equal number of $V_+$ and $V_-$ vertices in their interiors.
The endpoints of these $V(e)$ alternate between $V_+$ and $V_-$.
Therefore, the union of these $V(e)$, along with their endpoints, contains the same number of $V_+$ and $V_-$ vertices, counted with multiplicity.

On the other hand, the edges together form a cycle, so this union is a multiple of the full (counterclockwise) cyclic order.
Since $V_+$ is a larger set than $V_-$, the multiple must be zero.
This is a contradiction and means the graph is a path.
\end{proof}
\begin{proof}[Proof of Theorem \ref{mainthm}]
As already discussed, our knot and bands satisfy condition $(1)$ of Theorem \ref{crit} and $K*\mathcal{A}*B_1*\dots*B_k$ and $K*A_1*\dots*A_k*\mathcal{B}$ are unlinks for all $k$. By Proposition \ref{graphprop}, every $\mathcal{B}$ band is a fusion band to $K*\mathcal{A}$, so   $K*\mathcal{A}*\mathcal{B} = U_{1}$. This means that condition $(2)$ is satisfied as well, as (oriented) band moves change the number of components in a link by exactly one. Therefore, $K$ is doubly slice by Theorem \ref{crit}.
\end{proof}

\section{From Knots to Links}
We close with a remark about the case of links. For the odd pretzels with an even number of twist boxes, we have two component links instead of knots. In this case, there is some ambiguity in the definition of double slicing. One could either mean that the link is the cross section of an unknotted sphere in $S^4$, or a two component 2-unlink. The second definition, which we call {\bf strongly doubly slice}, is in some ways more natural, as it consists of a concatenation of two pairs of slice discs for the link. However, the first definition, which we call {\bf weakly doubly slice}, aligns more nicely with the obstructions coming from embeddings of branched covers into $S^4$, as the cyclic covers of $S^4$ branched over a two component 2-unlink are not $S^4$. However, strongly doubly slice implies weakly doubly slice, as one could tube the unknotted spheres from strong double slicing together to exhibit a weak double slicing. Therefore, one could obstruct strong double slicings by obstructing weak double slicings via a similar method to Issa and McCoy. Such a program, however, would fail to see the difference between the two properties. We give an example illustrating that these two notions are different:

\begin{Prop} \label{difference}
There exists a link that is weakly doubly slice but not strongly doubly slice.
\end{Prop}
\begin{proof}
For a strongly doubly slice link, each of the components must be doubly slice, as otherwise they could not individually be slices of unknotted spheres. Therefore, giving an example of a weakly doubly slice link with non-doubly slice components would suffice for a proof. Consider $L$, the $0$-framed $2$-cable on a slice but non-doubly slice knot $K \subset S^3$ (for example, the stevedore knot). Take a slice disk $D$ with boundary $K$, and form an interval sub-bundle of its trivial normal bundle. This interval sub-bundle is a $3$-ball whose boundary intersects $S^3_\epsilon$, a slight push-in of $S^3$ into $B^4$, in $L$. Therefore, $L$ is weakly doubly slice, but not strongly doubly slice.
\end{proof}

More generally, we conjecture that there exists a link with doubly slice components that is weakly doubly slice but not strongly doubly slice.

On the constructive side, there is an equivalent version of Theorem \ref{crit} for two component links, which gives strong double slicings for links. One could then see if a similar set of band attachments would give double slicings for the mutants of $P(a,-a,\dots,a,-a)$. In this case, we would need to attach two sets of $n-1$ bands, as the remaining $P(a,-a)$ would be a two component unlink. There are an equal number of $a$ and $(-a)$-twist boxes, so all twist boxes can be paired in both the clockwise and counterclockwise direction. Therefore, there is sometimes ambiguity as to which $n-1$ of the $n$ bands one should choose. The corresponding graph for $L$ has $2n$ vertices and two sets of $n$ edges for the clockwise and counterclockwise pairings, with one edge removed from each set of $n$.

For $L$ a mutant of $P(a,-a,\dots,a,-a)$, $L*\mathcal{A}*B_1*\dots*B_k$ and $L*\mathcal{B}*A_1*\dots*A_k$ are unlinks by the same logic as before. Therefore, it simply remains to show that the unlinks have the correct number of components. However, the cycle free property of the associated graph does not necessarily imply that the set of bands we choose is a double slicing. This is because after one attaches a set of $n-1$ bands, edge attachments to the remaining two unpaired vertices do not correspond to band attachments to the two components of $P(a,-a)$. Instead, attaching edges to either of the two remaining vertices corresponds to attaching bands to the same unknotted component of $P(a,-a)$. Because of this, the only mutant whose ribbon discs from this procedure concatenate properly is $P(a,-a,\dots,a,-a)$, which we conjecture is the only family.

\begin{Conj}
The only strongly doubly slice odd pretzel links are of the form $P(a,-a,\dots,a,-a)$ for some integer $a$.
\end{Conj}

We do not have as good of a sense for the weak double slicings, but the author's best guess is that these are the only ones.
\section*{Acknowledgements.} Thanks to my advisor Joshua Greene for helpful conversations in preparing this note and for indicating the line of the argument for Proposition \ref{graphprop}. Thanks to Joshua, Andrew Donald, and Duncan McCoy for pointing out the example in Proposition \ref{difference}. Thanks also to Fraser Binns, for showing this problem to me, and to Patrick Orson, for helpful comments.
\bibliographystyle{plain}
\bibliography{pretzel}

\begin{thebibliography}{1}

\bibitem{donald}
Andrew Donald.
\newblock Embedding {S}eifert manifolds in {$S^4$}.
\newblock {\em Trans. Amer. Math. Soc.}, 367(1):559--595, 2015.

\bibitem{issamccoy}
Ahmad Issa and Duncan McCoy.
\newblock Smoothly embedding {S}eifert fibered spaces in {$S^4$}.
\newblock
  \href{https://arxiv.org/abs/1810.04770}{\nolinkurl{https://arxiv.org/abs/1810.04770}},
  2018.

\bibitem{sch}
Martin Scharlemann.
\newblock Smooth spheres in {${\bf R}^4$} with four critical points are
  standard.
\newblock {\em Invent. Math.}, 79(1):125--141, 1985.

\end{thebibliography}

\end{document}